\documentclass[11pt]{article}
\usepackage{amsmath,amstext,amssymb,amsopn,amsthm}

\theoremstyle{plain}
\newtheorem{theo}{Theorem}[section]

\newtheorem{lem}[theo]{Lemma}
\newtheorem{prop}[theo]{Proposition}

\theoremstyle{definition}
\newtheorem{defi}[theo]{Definition}

\theoremstyle{remark}
\newtheorem{rem}{Remark}
\newtheorem*{rem*}{Remark}

\newcommand{\barint}{\rule[.036in]{.12in}{.009in}\kern-.16in \displaystyle\int}
\def\r{{\mathbb{R}}}

\def\rn{{\mathbb{R}^{n}}}

\def\K{\mathcal{K}}
\def\H{\mathcal{H}}
\def\W{\mathcal{W}}
\def\E{\mathcal{E}}
\def\diam{\mathrm{diam}\,}

\def\ind{\mathrm{ind}}
\def\vinf{{V^{(\infty)}}}
\def\vm{V^{(m)}}
\def\vo{V^{(0)}}
\def\mnb{\stackrel{m}{\sim}}
\def\grad{\langle \nabla f, Z\,\nabla f\rangle}

\newcommand{\pref}[1]{(\ref{#1})}

\numberwithin{equation}{section}

\usepackage{color}

\title{Poincar\'e Inequality and Haj{\l}asz-Sobolev spaces on nested fractals}

\author{Katarzyna Pietruska-Pa{\l}uba
\\
Institute of Mathematics\\
University of Warsaw\\
ul. Banacha 2\\
02-097 Warsaw, Poland\\
e-mail: {\tt kpp@mimuw.edu.pl}
\\
\\
Andrzej Stos\\
Clermont Universit\'e, Universit\'e Blaise Pascal\\
Laboratoire de Math\'ematiques, CNRS UMR 6620,\\
BP 80026, 63171 Aubi\`ere, France\\
e-mail: {\tt stos@math.univ-bpclermont.fr}
}

\begin{document}
\date{}
\maketitle

\abstract{
Given a nondegenerate harmonic structure, we prove a Poincar\'e-type inequality for
 functions in the domain of the Dirichlet form on nested fractals.
 We then study the Haj{\l}asz-Sobolev spaces on nested fractals.
 In particular, we describe how the "weak"-type gradient on nested fractals relates to the
 upper gradient defined in the context of general metric spaces. \\

}

{\small \noindent
{{\bf Keywords:} Nested fractals, Poincar\'{e} inequality\\
{\bf 2010 MS Classification:} Primary 46E35, Secondary 31E05, 28A80.}

\section{Introduction}

The interest in analysis on fractals arose from mathematical
physics, and dates back to the 80's of the past century. The first
object to be meticulously defined was the Kigami Laplacian on the
Sierpi\'nski gasket \cite{Kig-lap}, and, somehow in parallel, the
Brownian motion on the gasket \cite{Bar-Per}. Since then, we have
seen an outburst of papers focusing both on analytic and
probabilistic aspects of stochastic processes with fractal
state-space. The analytic approach, concerned mostly with
Dirichlet forms, their domains and generators, proved particularly
useful while constructing processes on fractals. On the other
hand, derivatives on fractals have been defined
\cite{Kig3,Kus2,Str5,Tep} and their properties studied. For an
account of results from that time, as well as an extended list of
references, we refer to  \cite{Kig-book} (analytic) and \cite{Bar}
(probabilistic).

In present paper, departing from the definition of the gradient on
nested fractals from \cite{Kus2,Tep}, we prove certain
Poincar\'e-type inequalities on nested fractals, for functions
belonging to the domain of the Brownian Dirichlet form (which can
be seen as a fractal counterpart of the Sobolev space
$W^{1,2}(\mathbb{R}^d)$). We will then  be concerned with
Poicar\'e-Sobolev spaces {and spaces of Korevaar-Schoen type,}
 and our analysis will be much in spirit of \cite{Kos-MMan}
and \cite{KST}.

In the last paper mentioned, the authors consider general metric
measure spaces equipped with a Dirichlet structure $({\cal E},
{\cal D}({\cal E}))$, much alike the nested fractals we consider.
However, in order to proceed, they make a standing assumption that
the intrinsic metric related to the Dirichlet structure,
\[
d_E(x,y)=\sup\{\phi(x)-\phi(y): \phi\in \Gamma; \;d\eta_{\mathbb{R}}(\phi,\phi)\leq d\mu\},
\]
where $\Gamma$ is a $\mu$-separating core of $\cal E$, induces the
topology equivalent to the initial one. This assumption fails for
fractals: the metric $d_E$ is degenerate there, see
\cite{Bar-Bas-Kum}, p.6. So, in order to extend the results from
\cite{KST}, one should either modify the definition of $d_E$ or
choose a different approach.

A discussion of gradients with connetcion to the Poincar\'e
inequality and relation between various function spaces can be
found  in the recent paper \cite{GKZ}. While three types of
gradients are considered, the one used in P.I. is the so called
{\it upper gradient}, a notion that depends on rectifiable curves.
In the context of nested fractals there may be no such curves at
all. Again, for a meaningful theory a different notion of gradient
should be considered.

We propose a hands-on approach based on discrete approximations of
nested fractals and Kusuoka gradients. By a limiting procedure,
the gradient can be reasonably defined for functions belonging to
the domain of the Dirichlet form, although it is usually hard to
decide whether the limit exists at a given point for functions
other than $m$-harmonic. This gradient can be used in
Poincar\'e-type inequalities and in defining variants of Sobolev
spaces on fractals.

We start with a local version of Poincar\'e inequality (P.I., for
short), which then yields a global P.I. on nested fractals. We
obtain inequalities of the form
\begin{equation}\label{eq-poinc-intro}
\barint_B |f-f_B|d\mu \leq C
r^{\frac{d_w}{2}}\left(\frac{1}{r^{d}}\int_{B(x_0, Ar)}\grad
d\nu\right)^{1/2},
\end{equation}
where $\mu$ is the $d$-dimensional Hausdorff measure on the
fractal, $\nu$ is the  Kusuoka energy measure on the fractal (see
Section \ref{sec-kus-meas} for a precise definition), $d_w$ is the
walk dimension of the fractal we are considering, $d$ its
Hausdorff dimension, and $\grad$ replaces the square of the norm
of the gradient. The measure $\nu$ is typically singular with
respect to the Hausdorff measure, but does not charge points.
Observe that in the Euclidean case we have $d_w=2$,
and so the scale function in P.I. will be linear as it should.
Poincar\'e inequalities involving  the Dirichlet energy measure in
a general setting have been investigated in the paper
\cite{Bar-Bas-Kum}, but that paper did not relate to the
definition of gradients on fractal sets.
{A choice, or even the existence,
of a gradient is not obvious on fractals.  We propose to
use a weak-type gradient with the energy measure (cf. Section 2
\ref{grad-nest}). }

As an application, in the second part of our paper, we compare
several possible definitions of Sobolev-type functions on
fractals. On metric spaces, several definitions of Sobolev-type
spaces have been considered (see e.g.  \cite{FHK},
\cite{Hajl-met}, \cite{KS}),  and nested fractals are of
particular interest in this context.  In present paper, we
introduce Poincar\'{e}-inequality based Sobolev spaces on fractals
and examine their relation with Korevaar-Schoen spaces and
Hajlasz-Sobolev spaces. While in a typical situation on metric
spaces the scaling factor in a Poincar\'e inequality is $r$, the
radius of a given ball, it turns out that on nested fractals it
doesn't yield an interesting inequality. To deal with relevant
Sobolev spaces, one should take into account the specific geometry
of the fractal and use a scaling factor $r^{d_w/2}$, as in
\pref{eq-poinc-intro}. For  some preliminary relations between
Hajlasz-Sobolev and Korevaar-Schoen Sobolev spaces on fractals we
refer to  a paper by Hu \cite{Hu}.

\section{Preliminaries}

We use $C$ or $c$ to denote a positive constant depending possibly
on the fractal set, whose exact value is not important for our
purposes and which may change from line to line. We will write
$f\asymp g$ (on a set $D$) if there exists a constant $C>0$ such
that for every $x\in D$ one has $C^{-1}g(x) \le f(x) \le Cg(x)$.
For an $m$-integrable function $f$ and a set $A$ of finite measure
we adopt the notation $f_A=\barint_Afdm=\frac{1}{m(A)}\int fdm.$

\subsection{Nested fractals}\label{nested fractals}

The framework of nested fractals is that of Lindstr\"om \cite{Lin}.
Suppose that $\phi_1,...,\phi_M,$ $ M\geq 2,$ are  similitudes of
$\mathbb{R}^N$ with a common scaling factor $L>1$. When $A\subset
\rn$, then we write $\Phi(A)$ for $\bigcup _{i=1}^M \phi_i(A),$
and $\Phi^m$ for $\Phi$ composed $m$ times. There exists a unique
nonempty compact set (see \cite{Fal}, \cite{Lin}) $\K\subset
\mathbb{R}^N$ such that
\begin{equation}\label{ka}
\K =\bigcup_{i=1}^M\phi_i(\K)=\Phi(\K).
\end{equation}
It is called the {\em self-similar fractal} generated by the
family of similitudes $\phi_1,...\phi_M.$ Since the set $\K$ has a
finite nonzero diameter, for simplicity we can and will assume
that $\diam\K=1.$

Each of the mappings $\phi_i$ has a unique fixed point $v_i.$
Such a point is called {\em an essential fixed point} if there
exists another fixed point $v_j$ such that for some
transformations $\phi_k, \phi_l$ one has
$\phi_k(v_i)=\phi_l(v_j).$ The set of all essential fixed points
will be denoted by $V^{(0)}=\{v_1,...,v_r\}.$ For $m=1,2,...$ we
set $\vm = \Phi^m(\vo)$ and $ \vinf = \bigcup_{m\ge 0} \vm$. For
nondegeneracy, we assume that $r= \# \vo \ge 2$.

\smallskip

The system $\{\phi_1,...,\phi_M\}$ {is said to satisfy}
 the {\em open set
condition} if there exists an open, nonempty set $U$ such that
$\Phi(U)\subset U$ and for all $i\neq j$ one has $\phi_i(U)\cap
\phi_j(U)=\emptyset.$   If the open set condition is satisfied,
then the Hausdorff dimension of the self-similar fractal $\mathcal
K$ is equal to $d=d({\mathcal K})=\frac{\log M}{\log L}.$ By $\mu$
we denote the $d$-dimensional Hausdorff measure on $\K$ normalized
so that $\mu(\K)=1$.

For $m\geq 1,$ by a {\em word } of length $m$ we mean a sequence
$w=(w_1,...,w_m)\subset\{1,...,M\}^m.$ Collection of all words of
length $m$ is denoted by $\mathcal W_m;$ $\mathcal
W_*=\bigcup_{m\geq 1}\mathcal W_m$ consists of all words of finite
length, $\mathcal W$ is the collection of all infinite words. When
$w\in\mathcal W_*$ is a finite word, then $|w|$ denotes its
length. If  $w\in\mathcal W$ is  an infinite word, then $[w]_m$
denotes its restriction to first  $m$ coordinates, i.e. for
$w=(w_1,w_2,...),$ $[w]_m=(w_1,...,w_m).$ When
$w=(w_1,\ldots,w_m)$ is given, then we will use the notation
$\phi_w=\phi_{w_1}\circ\cdots\circ\phi_{w_m},$ and for a set $A,$
$A_w=\phi_w(A).$

\begin{defi}\label{sxy}
Let $m\geq 1$.
\begin{itemize}
\item[(1)] An {\em $m$-simplex} is any  set of the form $\phi_w(\K)$ with $w\in\mathcal W_m$
($m$-simplices are just scaled down copies of $\K).$ The
collection of all $m$-simplices will be denoted by ${\mathcal
T}_m.$ The $0$-simplex is just $\K.$

\item[(2)] For an $m$-simplex $S=\phi_w(\K),$ $w\in\mathcal W_m,$ let $V(S)=\phi_w(V^{(0)})$ be the set of its vertices. An {\em $m$-cell} is any of the sets $\phi_w(V^{(0)}).$ Two points $x,y\in V^{(m)}$ are called {\em $m$-neighbors}, denoted $x \mnb y$, if they belong to a common $m$-cell.

\item[(3)] If $\Delta\in {\mathcal T}_m$, $m\geq 1$, we denote by $\Delta^*$ the union of $\Delta$ and all the adjacent $m$-simplices, and by $\Delta^{**}$ -- the union of $\Delta^*$ and all $m$-simplices adjacent fo $\Delta^*.$

\item [(4)] For any $x\in \K \setminus \vinf$ and $m\geq 1,$ set $\Delta_m(x)$ to be the unique  $m$-simplex that contains $x$.

\item[(5)]  For any $x,y\in \K\setminus V^\infty$, define $\mathrm{ind}(x,y) = \min\{m\ge 1:\, \Delta_m(x)\cap \Delta_m(y)=\emptyset \}$.  When $\mathrm{ind}(x,y) = n$, we set $S(x,y)=\Delta_{n-1}(x)\cup \Delta_{n-1}(y)$.

\item[(6)] When an $m$-simplex $\Delta= \K_w =\phi_w(\K),$ $w\in\mathcal W_m$ is given and $\tilde w\in\mathcal W_n$ is another finite word, then by $\Delta_{\tilde w}$ we denote the $(m+n)$-simplex $\phi_{w\tilde w}(\K).$

\end{itemize}
\end{defi}

From now on we will assume that for every $S,T\in {\mathcal T}_m,$
$m\ge 1$, with $S\not=T$, one has $S\cap T=V(S)\cap V(T)$
(nesting).  Define the graph structure $E_{(1)}$ on $V^{(1)}$ as
follows: { we say that $(x,y)\in E_{(1)},$ if $x$ and $y$ are
$1$-neighbors.} Then we require the graph $(V^{(1)},E_{(1)})$ to
be connected. For $x,y\in V^{(0)}$, let $R_{x,y}$ be the
reflection in the hyperplane bisecting the segment $[x,y].$ Then
we stipulate that
\[
\forall _{i\in\{1,...,M\}}\forall _{x,y\in V^{(0)},\;x\neq y}
\exists _{j\in\{1,...,M\}}\;\;
R_{x,y}(\phi_i(V^{(0)}))=\phi_j(V^{(0)})
\]
(natural reflections map 1-cells  onto 1-cells).

The self-similar fractal $\K$ is called a {\em nested fractal}, if
it satisfies the above open set condition, nesting, invariance
under local isometries, and the connectivity assumption.

\medskip

\noindent Part of our results will require the following Property
{\bf (P)} of the fractal:

\medskip

{\bf\noindent Property (P).} There exist $\alpha>0$ such that for
all $n=1,2,...$ and  $x,y$ -- nonvertex points such that
$y\in\Delta_n^*(x)\setminus \Delta_{n+1}^*(x)$ one has
\begin{equation}\label{pppp}
\rho(x,y)\geq \frac{\alpha}{L^n}.
\end{equation}

\medskip

\begin{rem}\label{num}
Property {\bf (P)} holds true for nested fractals such that
the similitudes $(\phi_i)_{i=1,...,M}$ have the same unitary part.
This class of fractals contains the well-knows examples such as the
Sierpi\'{n}ski gaskets, snowflakes, the Vicsek set etc. Proof of
this statement is given in the Appendix.
\end{rem}

Clearly, if $\mathrm{ind}(x,y) = n,$ then $\Delta_{n-1}(x)\cap
\Delta_{n-1}(y) \not=\emptyset$. These sets either coincide or are
adjacent (ie. they meet at exactly one point). Moreover, under
Property {\bf (P),} the index $\mbox{ind}\,(x,y)$ is closely
related to the {Euclidean} distance of $x,y.$

\begin{lem}\label{ind-n}
1. For any fixed $x\in \K\setminus \vinf$ and $n\ge 2$, one has
\begin{equation}\label{eq-ind-n}
 \{ y:\, \mathrm{ind}(x,y) = n  \}  =
\Delta^*_{n-1}(x)\setminus \Delta^*_n(x).
\end{equation}
2.  Assume  additionally that  the fractal $\K$ satisfies the
property {\bf (P)}.  If $  \mathrm{ind}\,(x,y) = n$ then
\begin{equation}\label{ind-dist}
\rho(x,y) \asymp L^{-n},
\end{equation}
$\rho(x,y)$ being the Euclidean distance.
\end{lem}
\begin{proof}
Fix $x\in \K\setminus \vinf$ and $n\ge 2$. Observe that  $y \in
\Delta^*_n(x)$ if and only if
$\Delta_n(x)\cap\Delta_n(y)\neq\emptyset,$ which is equivalent to
$\mbox{ind}\,(x,y)\geq n+1.$ Since $\{\Delta_{n}(x)\}_n$ is a
decreasing sequence of sets, \pref{eq-ind-n} follows.  Relation
\pref{ind-dist} follows from \pref{eq-ind-n} and Property {\bf
(P)}.
\end{proof}

\subsection{Gradients of nested fractals}\label{grad-nest}
To proceed, we need to define the gradient. The material in this
section is classic and follows mainly \cite{Kig-book} and
\cite{Tep}. For other results concerning gradients on fractals we
refer to \cite{Kus2,Kig3,Str5}.
\subsubsection{Nondegenerate harmonic structure on $\K$} \label{nondegstruc}
Suppose that $\K$ is the nested fractal associated with the system
$\{\phi_1,...,\phi_M\}.$ Let $A=[a_{x,y}]_{x,y\in\vo}$ be a {\em
conductivity matrix} on $\vo,$ i.e. a symmetric real matrix with
nonnegative off-diagonal entries and such that for any $x\in\vo,$
$ \sum_{y\in \vo}a_{x,y}=0.$ For $f:\vo\to\r,$ set ${\mathcal
E}_A^{(0)}(f,f)=\frac{1}{2}\sum_{x,y\in\vo}a_{x,y}(f(x)-f(y))^2.$
Then we define two operations:
\begin{itemize}
\item[(1)] {\em Reproduction}. For $f\in C(V^{(1)})$ we
let
\[\widetilde{\mathcal E}^{(1)}_A(f,f)=\sum_{i=1}^M {\mathcal
E}^{(0)}_A(f\circ \phi_i, f\circ\phi_i).\] The mapping ${\mathcal
E}_A^{(0)}\mapsto \widetilde{\mathcal E}^{(1)}_A$ is called the
{\em reproduction map} and is denoted by $\mathcal R$.\\
\item[(2)] {\em Decimation}. Given a symmetric form $\mathcal E$ on
$C(V^{(1)}),$ define its restriction to $C(V^{(0)}),$ ${\mathcal
E}_{V^{(0)}},$ as follows. Take $f:V^{(0)}\to\mathbb{R},$ then set
\[{\mathcal E}|_{V^{(0)}}(f,f)=\inf\{{\mathcal E}(g,g):
g:V^{(1)}\to\mathbb{R}\mbox{ and } g|_{V^{(0)}}=f\}.\] This
mapping is called the {\em decimation map} and will be denoted by
${\mathcal D}e.$
\end{itemize}
Let $\bf G$ be the symmetry group of $V^{(0)},$ i.e. the group of
transformations generated by symmetries $R_{x,y},$ $x,y\in
V^{(0)}.$ Then we have (\cite{Lin}, \cite{Sab}):
\begin{theo}\label{titi}
Suppose $\mathcal K$ is a nested fractal. Then there exists a
unique number $\rho=\rho({\mathcal K})>1 $ and a unique, up to a
multiplicative constant, irreducible conductivity matrix $A$ on
$V^{(0)},$ invariant under the action of $\bf G,$ and such that
\begin{equation}\label{rho}
({\mathcal D}e\circ{\mathcal R})({\mathcal E}^{(0)}_A) =
\frac{1}{\rho}\,{\mathcal E}^{(0)}_A.
\end{equation}
\end{theo}
$A$ is called the {\em symmetric nondegenerate harmonic structure
on $\mathcal K$}. By analogy with the electrical circuit theory,
$\rho$ is called the {\em resistance scaling factor} of $\mathcal
K.$ The number  $d_w = d_w({\mathcal K}) \stackrel{def}{=}
\frac{\log(M\rho)}{\log L}
> 1$ is called the {\em walk dimension} of $\mathcal K$.
For further use, note that $\rho=L^{d_w-d}$.

\subsubsection{The canonical Dirichlet form on $\K$}
Suppose $A$ is the nondegenerate harmonic structure on $\K.$
Define ${\mathcal E}^{( 0)}={\mathcal E}^{( 0)}_A,$ then let
\[
\widetilde{\mathcal E}^{(m)}(f,f)= \rho^m \sum_{|w|=m}{\mathcal
E}^{(0)} (f\circ\phi_w,f\circ\phi_w),\qquad f\in C(\vm).
\]
The sequence $\widetilde{\mathcal E}^{( m )}$ is nondecreasing,
i.e. for every $f:V^{(\infty)}\to\mathbb{R},$ one has
\[\widetilde{\mathcal E}^{(m)}(f,f)\leq \widetilde{\mathcal E}^{(m+1)}(f,f), \quad m=0,1,2,\ldots\]
Set $\widetilde{\mathcal D} = \{f:V^{(\infty)}\to\mathbb{R}:\sup_m
\widetilde{\mathcal E}^{(m)}(f,f)<\infty\}$ and for $f\in
\widetilde{\mathcal D}$
\begin{equation}\label{jeden}
\widetilde{\mathcal E}(f,f)=\lim_{m\to\infty}\widetilde{\mathcal
E}^{(m)}(f,f).
\end{equation}
Further, ${\mathcal D}={\cal D}(\cal E)=\{f\in C({\mathcal K}):
f|_{V^{(\infty)}}\in\widetilde{\mathcal D}\},$ ${\mathcal E
}(f,f)=\widetilde{\mathcal E}(f|_\vinf, f|_\vinf)$ for
$f\in{\mathcal D}.$

Then $({\mathcal E }, {\mathcal D} )$ is a regular local Dirichlet
form on $L^2({\mathcal K},\mu),$ which agrees with the group of
local symmetries of ${\mathcal K.}$ This Dirichlet form is also
called `the Brownian Dirichlet form on $\K$', and will be
essential in defining the gradient. It satisfies the following
scaling relation: for any $f\in {\mathcal D}$,
\begin{equation}\label{scaling-form}
{\mathcal E}(f,f)=\rho^m\sum_{w\in\mathcal W_m} {\mathcal
E}(f\circ\phi_w,f\circ\phi_w).
\end{equation}

\subsubsection{Harmonic functions on $\K$ and energy measure}\label{sec-kus-meas}
\begin{defi}\label{def_harmonic}
 Suppose $f:V^{(0)}\to \mathbb{R}$ is given. Then $h\in
{\mathcal D}({\mathcal E})$ is called {\em harmonic} on $\mathcal
K$ with boundary values $f,$ if ${\mathcal E}(h,h)$ minimizes the
expression ${\mathcal E}(g,g)$ among all $g\in {\mathcal
D}({\mathcal E}) $ such that $g|_{V^{(0)}}=f.$  The unique
harmonic function that agrees with $f$ on $\vo$ will be denoted by
$Hf.$
\end{defi}

Denote by $\mathcal H$ the space of all harmonic functions on
$\K.$ It is an $r$-dimensional linear space, which can be equipped
with the norm
\[\|h\|^2_{\mathcal H}={\mathcal E}(h,h)+(\sum_{x\in \vo} h(x))^2.\]
Further, $\widetilde{\mathcal H}$ denotes the orthogonal
complement in $\mathcal H$ of the (one-dimensional) subspace of
constant functions, and let $\widetilde {P}:{\mathcal H}\to
\widetilde {\mathcal H}$ be the orthogonal projection onto
$\widetilde{\mathcal H}$. The norm on $\widetilde{\H}$ is given by
$\|h\|^2={\mathcal E}(h,h)$ (note that $\|\cdot\|$ is a seminorm
on $\H,$ vanishing on constant functions), and the corresponding
scalar product on $\widetilde {\mathcal H}$ will be denoted by
$\langle \cdot,\cdot\rangle.$

Next, for $i=1,...,M,$ we define the map $M_i:{\H}\to{\H}$ by
$M_ih=h\circ\phi_i,$ and $\widetilde {M}_i:\widetilde{\H}\to
\widetilde{\H}$ by $\widetilde{M}_i= \widetilde{P}\circ M_i.$ From
the scaling relation \pref{scaling-form} we deduce that for
$h\in\widetilde{\H}$ and $m\geq 0,$
\begin{equation}\label{scaling-harm}
\|h\|^2= \rho^m\sum_{|w|=m} \|\widetilde{M}_{w}h\|^2,
\end{equation}
where by $\widetilde {M}_w$  we have denoted
$\widetilde{M}_{w_m}\circ\cdots\circ\widetilde{M}_{w_1}h=\widetilde
P(h\circ\phi_w).$

For $f\in {\mathcal D},$ let's define the {\em energy measure
associated with $f$} as the measure whose value on any given
$m$-simplex $\K_w=\K_{w_1\ldots w_m}$ is equal to
\begin{equation}\label{en-meas-def}
\nu_f(\K_w)=\rho^m\,{\mathcal E}(f\circ \phi_w,f\circ\phi_w).
\end{equation}
When $h\in \H$ is a harmonic function and $w\in \W_m$, then
$\nu_h(\K_w)=\rho^m\|M_wh\|^2.$ Let $h_1,...,h_{r-1}$ be an
orthonormal basis in $\widetilde\H.$ Then the expression
\begin{equation}\label{kus-measure}
\nu\stackrel{def}{=}\sum_{i=1}^{r-1}\nu_{h_i}
\end{equation}
does not depend of the choice of the orthonormal basis and its
value on an $m$-simplex $\K_w$ is equal to
\[\nu(\K_w)= \rho^m \mbox{Tr}\, \widetilde{M}^*_w\widetilde{M}_w.\]
The measure given by \pref{kus-measure} is called the {\em Kusuoka
measure}, or the {\em energy measure} on $\K.$ This measure has no
atoms, and typically is singular with respect to the measure
$\mu.$

\subsubsection{Gradients}
When $x\in\K$ is a nonlattice point, then $x$ has a unique
address: it is an (infinite) sequence $w=w_1w_2\ldots$ such that
$x=\bigcap_{m=1}^\infty \K_{[w]_m}$ (recall that we have denoted
$[w]_m=(w_1...w_m)$). For such a nonlattice point, let
\begin{equation}\label{labelek}
Z_m(x)=\left\{\begin{array}{ll}
\displaystyle\frac{\widetilde{M}^*_{[w]_m}\widetilde{M}_{[w]_m}}{\mbox{Tr}
\,\widetilde{M}^*_{[w]_m}\widetilde{M}_{[w]_m}} & \mbox{if } \mbox{Rank}\,\widetilde{M}_{[w]_m}>0;\\
0 & \mbox{otherwise}. \end{array}\right.
\end{equation}
It can be shown that $Z_m(\cdot)$ is a bounded, matrix-valued
martingale with respect to $\nu,$ and as such it is convergent
$\nu$-a.s. to an integrable function $Z(\cdot).$

For a nonlattice point $x$ with address $w,$ set
\[\nabla_{m}f(x)=\widetilde{M}_{[w]_m}^{-1}(\widetilde{P} H)(f\circ
\phi_{[w]_m}),\;\;\; m=1,2,...,\] then  the gradient of $f$ at
point $x$ is the element of $\widetilde\H$ given by
\[{\nabla} f(x)=\lim_{m\to\infty}\nabla_{m}f(x),\]
provided the limit exists.   For the discussion of the `pointwise
gradients' and their properties we refer to \cite{Tep},
\cite{Pel-Tep} and \cite{Hin}. But even if the pointwise limits of
$\nabla_{m}$ are not known to exist, we do know (see \cite{Kus2},
Lemmas 3.5 and 5.1, and also the discussion in \cite{Tep}, p. 137)
that when $f\in {\mathcal D},$ then there exists a measurable
mapping $Y(\cdot, f)$ such that
\begin{equation}\label{grad-form}
\mathcal E(f,f)=\int_{\K}\langle Y(\cdot, f),Z(\cdot)Y(\cdot,
f)\rangle d\nu(\cdot).
\end{equation}
With an abuse of notation, we will write $\nabla f$ for the object
$Y(\cdot,f),$ which is defined $\nu$-a.e. When we will use the
pointwise value, it will be clearly indicated.

\begin{defi}\label{harm-m}
\begin{enumerate}
\item A continuous function $f:\K\to\r$ is called $m$-{\em harmonic} if $f\circ \phi_w$ is harmonic for any $w=(w_1...w_m)\in \W_m.$
\item
There exists a unique $m$-harmonic function with given values at
points from $V^{(m)}.$ For a continuous function $f$ on $\K,$ by
$H_mf$ we denote the unique $m$-harmonic function that agrees with
$f$ on $ V^{(m)}.$
\end{enumerate}
\end{defi}

\begin{rem}
When $f$ is $m$-harmonic, then for any nonlattice point $x\in\K$
with address $w\in\W_\infty$ one has
\[\nabla_{m}f(x)=\nabla_{m+n}f(x)\]
for any $n\geq 0,$ and so $\nabla f(x)$ exists at nonlattice
points (which are of full $\nu$-measure); note also that $\nabla_m
f - f$ (and thus also $\nabla f-f$) is constant inside each $\K_w$
with $|w|=m$.
\end{rem}

\section{Poincar\'e inequality on nested fractals}

Poincar\'e inequalities on nested fractals that one can find in
the literature (see e.g. \cite{Bar-Bas-Kum} and its references)
are usually written in the form
\begin{equation}\label{poinc-bbk}
{\int}_B |f-f_B|^2d\mu \leq c \Psi(R)\, {\int}_B d\Gamma(f,f),
\end{equation}
where $B$ is a ball of radius $R,$
$\Psi:\mathbb{R}_+\to\mathbb{R}_+$ is a scale function (most
commonly, $\Psi(R)=R^\sigma),$ and $\Gamma(f,f)$ is the energy
measure associated with the Brownian Dirichlet form on fractals.

Poincar\'e inequalities $P(q,p),$ on a metric measure space
$(X,\rho,\mu),$ with a doubling measure $\mu$ and another Radon
measure $\nu,$ are similar in spirit, but involve usually {\em two
} functions. One says that a pair of measurable functions $(f,g)$
satisfies the $(q,p)$-Poincar\'e inequality, when
\begin{equation}\label{poinc-clas}
\left({\barint}_B |f-f_B|^qd\mu\right)^{1/q}\leq C R
\left({\barint}_{\sigma B}|g|^pd\nu\right)^{1/p},
\end{equation}
where $\sigma\geq 1$ is a given number, and $\sigma B$ denotes the
ball concentric with $B,$ but with radius $\sigma$ times the
radius of $B.$ For an account of Poincar\'e inequalities in metric
spaces, we refer mainly to \cite{Haj-Kos}, and also to
\cite{Hajl-met}.

Poincar\'{e} inequalities on nested fractals we will be concerned
with will be variants of  two-weight inequalities.. The measure
$\mu$ appearing on the left-hand side will be the Hausdorff
measure on $\mathcal K$, while the measure $\nu$ on the right-hand
side will be the Kusuoka energy measure. Recall that the measure
$\nu$ in most cases is not absolutely continuous with respect to
$\mu.$ The difference from the classical case  is that the
integral on the right-hand side will not be a barred integral with
respect to the measure $\nu,$ but it will be  divided  by the
measure $\mu$ of the underlying set.

We start with a fractal version of Poincar\'e inequality -- where
balls are replaced  with simplices. This version does not require
property {\bf (P)} of the underlying fractal. The precise
statement reads as follows.

\begin{theo}\label{najwazniejsze}
Let  $f\in {\mathcal D}({\mathcal E}),$ and let $\Delta$ be any
$m$-simplex, $m\geq 0.$ Then we have
\begin{eqnarray}\label{poincfrac}
\barint_\Delta |f(x) - f_\Delta| d\mu(x) &\le& C\,
(\diam\Delta)^{d_w/2} \left( \frac{1}{\mu(\Delta^*)}
\int_{\Delta^*} \grad \,d\nu \right)^{1/2}\nonumber\\
&\leq& C L^{-md_w/2}\left(L^{-md}\int_{\Delta^*}\langle \nabla f,
Z\nabla f\rangle\;d\nu\right)^{1/2},
\end{eqnarray}
where $\Delta^*$ denotes the union of $\Delta$ and all
$m$-simplices adjacent to $\Delta$.
\end{theo}

The proof will be given later on. Now, we start with a local
version of Poincar\'e inequality for adjacent lattice points.

\begin{prop}\label{prop1}
Suppose $f\in{\mathcal D}({\mathcal E}),$ and let
$x\stackrel{m}{\sim} y.$ Let $\K_w$ be the $m$-simplex that
contains both points $x,y,$ with address $w\in\mathcal W_m.$  Then
\begin{equation}\label{lem1-eq1}
|f(x)-f(y)|^2\leq C( \diam \K_w)^{d_w-d} \int_{\K_w} \grad d\nu.
\end{equation}
\end{prop}

\begin{proof}
Set $c(x,y)=a_{x'y'}^{-1}$ where $x',y'\in\vo$ are such that
$x=\phi_w(x')$ and $y=\phi_w(y')$ (the matrix  $A=[a_{x,y}]$ was
introduced in Section \ref{nondegstruc}). Then we have:
\begin{eqnarray*}
|f(x)-f(y)|^2 & \leq&
c(x,y)) \sum_{u,v\in \vo} a_{uv}|f\circ\phi_w(u)-f\circ\phi_w(v)|^2 \\
&=& c(x,y)) {\mathcal E}^{(0)}(f\circ \phi_w,f\circ \phi_w)
\leq {\mathcal E}(f\circ\phi_w,f\circ\phi_w)\\
&=& c(x,y)\int _\K \langle \nabla (f\circ \phi_w), Z \,\nabla
(f\circ \phi_w) \rangle d\nu\\
&\leq & c_1 \int _\K \langle \nabla (f\circ \phi_w), Z \,\nabla
(f\circ \phi_w) \rangle d\nu,
\end{eqnarray*}
where $c_1=\sup \{a_{x,y}:x,y\in V^{(0)}\}.$

 Since
$\mbox{diam}\,\K_w=L^{-m}$, the scaling relation from Lemma
\ref{lemscal} below gives the desired statement.
\end{proof}

\begin{lem}\label{lemscal}
Let $f\in{\mathcal D}({\mathcal E}),$ and let $\K_w$ be an
$m$-simplex. Then
\begin{equation}\label{toprove}
\int_{\K}\langle \nabla (f\circ \phi_w), Z \,\nabla (f\circ
\phi_w) \rangle d\nu= L^{-m(d_w-d)} \int_{\K_w} \langle \nabla f,
Z \,\nabla f \rangle d\nu
\end{equation}
\end{lem}
\begin{rem}
The right hand side of \eqref{toprove} is well-defined since $
\langle \nabla f, Z \,\nabla f \rangle$ exists $\nu$-a.e. and
$\int_\K  \langle \nabla f, Z \,\nabla f \rangle d\nu<\infty, $
see Theorem 4 of \cite{Tep}.
\end{rem}
\begin{rem}
While $\nu(\K_w)$ depends in general on $w$, the scaling factor on
the right hand side of \eqref{toprove} depends only on $m=|w|$.
Thus, the lemma is not tantamount to a simple change of variables
but reflects an interplay between $\nabla f$ and $Z$.
\end{rem}
\begin{proof}
\noindent\textsc{Step 1.} Assume that $f$ is $m$-harmonic. Then
$\nabla f(y)$ exists at all nonlattice points $y$ and $\nabla f
(y)=\nabla_m f(y).$ Observe that $\nabla_mf(\cdot)$ is constant
($\nu$-a.e.) inside each $m$-simplex $\K_w$ and that it differs
there from $M_w^{-1}f(\cdot)$ by a constant only. It follows that
\begin{eqnarray*}
\int_{\K_w}\grad d\nu = \int_{\K_w}\langle\nabla_m
f,Z\,\nabla_mf\rangle\,d\nu= \lim_{n\to\infty} \int_{\K_w}\langle
\nabla_mf,Z_n \nabla_mf \rangle d\nu.
\end{eqnarray*}
To justify the last statement, observe that the random variables
$X_n=\langle\nabla_m f,Z_n\nabla_mf\rangle$ converge to
$X=\langle\nabla_m f,Z\nabla_mf\rangle$ in $L^1(\K,d\nu).$ This is
so because $X_n\geq 0,$ $X_n\to X$ in measure $\nu$ and
\[\int_\K X_n\,d\nu =\int_\K \langle\nabla_m
f,Z_n\nabla_mf\rangle\, d\nu={\cal E}(H_nf,H_nf)\to {\cal
E}(f,f)=\int_K X\,d\nu.\] The convergence in $L^1(\K,d\nu)$
follows then from Scheff\'{e}'s theorem.

For short, let us write $F=\nabla_m f\in\widetilde{\mathcal H}.$
Let $n>m$ be fixed, and let
${\underline{i}}=(i_{m+1},\ldots,i_n)\in\W_{n-m}$ so that
$w{\underline{i}}\in\W_n$. $Z_n$ is constant on $n$-simplices and,
once $n>m,$ we have
\begin{eqnarray*}
\int_{K_w} \langle F, Z_n\,F\rangle \,d\nu &=&
\sum_{|{\underline{i}}|=n-m}\int_{\K_{w{\underline{i}}}} \langle
F, Z_n\,F\rangle \,d\nu
\\ &=& \sum_{|{\underline{i}}|=n-m} \frac{\|\widetilde{M}_{w{\underline{i}}}F\|^2}{\text{Tr}
(\widetilde{M}^*_{w{\underline{i}}} \widetilde
M_{w{\underline{i}}})} \cdot L^{n(d_w-d)}
\mbox{Tr}(\widetilde{M}^*_{w{\underline{i}}} \widetilde
M_{w{\underline{i}}})
\\&=&
L^{n(d_w-d)}
\sum_{|{\underline{i}}|=n-m}\|\widetilde{M}_{w{\underline{i}}}F\|^2
\\&=&
L^{n(d_w-d)}\sum_{|{\underline{i}}|=n-m}{\mathcal E}(F\circ
\phi_{w{\underline{i}}}, F\circ \phi_{w{\underline{i}}}).
\end{eqnarray*}
From the scaling property of $\cal E$,
\begin{eqnarray*}
\sum_{|{\underline{i}}|=n-m}\E(F\circ \phi_{w{\underline{i}}},
F\circ \phi_{w{\underline{i}}}) &=&
\sum_{|{\underline{i}}|=n-m}  \E((F\circ\phi_w)\circ \phi_{{\underline{i}}}, (F\circ \phi_w)\circ \phi_{{\underline{i}}}) \\
&=& L^{-(n-m)(d_w-d)}{\mathcal E}(F\circ\phi_w,F\circ\phi_w).
\end{eqnarray*}
We know that $F\circ\phi_w$ and $f\circ\phi_w$ differ by a
constant only, so that
\[{\mathcal E}(F\circ\phi_w,F\circ\phi_w)={\mathcal E}(f\circ\phi_w,f\circ\phi_w).\]
Piecing everything together, we obtain
\begin{eqnarray*}
\int_{\K_w}\langle F, Z_n  F\rangle\,d\nu &=& L^{m(d_w-d)}{\mathcal E}(f\circ\phi_w,f\circ\phi_w)\\
&=&L^{m(d_w-d)}\int_{\K} \langle \nabla(f\circ\phi_w),
Z\,\nabla(f\circ\phi_w)\rangle d\nu.
\end{eqnarray*}
The right-hand side does not depend on $n,$ thus we can pass with
$n$ to infinity, obtaining \pref{toprove}.

\medskip

 \noindent \textsc{Step 2.} Let now $f$ be $n$-harmonic,
with $n>m.$ Then
$\K_w=\bigcup_{|{\underline{i}}|=n-m}\K_{w{\underline{i}}}$ and
\begin{equation}\label{pojutrze}
\int_{\K_w}\grad
d\nu=\sum_{|{\underline{i}}|=n-m}\int_{\K_{w{\underline{i}}}}\grad
d\nu.
\end{equation}
To each of the integrals on the right-hand side of \pref{pojutrze}
we apply Step 1, obtaining
\begin{eqnarray*}
\pref{pojutrze}&{=}& L^{n(d_w-d)}\sum_{|{\underline{i}}|=n-m}
\int_\K\langle\nabla(f\circ\phi_{w{\underline{i}}}), Z\nabla(f\circ\phi_{w{\underline{i}}})\rangle d\nu\\
&=& L^{m(d_w-d)}\sum_{|{\underline{i}}|=n-m}L^{(n-m)(d_w-d)}
\E((f\circ\phi_w)\circ\phi_{{\underline{i}}},
(f\circ\phi_w)\circ\phi_{{\underline{i}}}),
\end{eqnarray*}
which is, from the scaling property, equal to
\[L^{m(d_w-d)}{\mathcal E}(f\circ\phi_w. f\circ\phi_w)= L^{m(d_w-d)}\int_\K\langle\nabla(f\circ\phi_w), Z\,\nabla(f\circ\phi_w)\rangle\,d\nu.\]

\noindent\textsc{Step 3}. Let now $f$ be any function from
${\mathcal D}({\mathcal E}).$ Then
\[{\mathcal E}(f,f)=\lim_{n\to\infty}{\mathcal E}(H_nf, H_nf)\]
and
\begin{equation}\label{pojutrze1}
{\mathcal E}(f\circ\phi_w,
f\circ\phi_w)=\lim_{n\to\infty}{\mathcal E}(H_n(f\circ\phi_w),
H_n(f\circ\phi_w)).
\end{equation}
From Step 2 we have: for $n\geq m,$
\begin{equation}\label{liii}
\int_\K\langle\nabla_n(H_n f\circ\phi_w), Z
\,\nabla_n(H_nf\circ\phi_w)\rangle\,d\nu
=L^{-m(d_w-d)}\int_{\K_w}\langle\nabla(H_nf),Z\,\nabla(H_nf)\rangle\,d\nu,
\end{equation}
and the assertion follows from the limiting procedure:  the
left-hand side of \pref{liii} is equal to ${\mathcal
E}(H_nf\circ\phi_w,H_nf\circ\phi_w)
\stackrel{n\to\infty}{\longrightarrow}{\mathcal E}(f\circ\phi_w,
f\circ\phi_w).$ As to the right-hand side, since
$\nabla(H_nf)=\nabla_nf,$ and $\nabla_nf$ converges to $\nabla f$
in the seminorm $\left(\int_\K\langle\cdot, Z\cdot\rangle
d\nu\right)^{1/2},$ we also have the convergence in the restricted
seminorm $\left(\int_{\K_w}\langle\cdot, Z\cdot\rangle
d\nu\right)^{1/2},$ which gives the desired convergence.
\end{proof}

\noindent From Proposition \ref{prop1} we derive the local
Poincar\'e inequality for nonlattice points.

\begin{theo}\label{poincpoint}  Supopose that $\K$ satisfies property {\bf (P)}.
Let  $f\in {\mathcal D}({\mathcal E})$ and  $x,y\in {\mathcal
K}\setminus V^{(\infty)}.$ Then
\[
|f(x)-f(y)|^2\leq
C\rho(x,y)^{d_w}\frac{1}{\mu(S(x,y))}\int_{S(x,y)} \grad d\nu.
\]
where $S(x,y)$ was introduced in Definition \ref{sxy} (6).
\end{theo}

\begin{proof}
\textsc{Step 1.} Suppose $z\in V^{(m)}$ is a vertex of $\Delta\in
{\mathcal  T}_m$ and let $y\in \mbox{Int}\,\Delta.$ Then one finds
a chain $z=z_0,z_1,...,z_k\to y$ such that for all $k=1,2,...$ the
points $z_{k-1}$ and $z_k$ are $(m+k)$-neighbors. Denote by
$\Delta(z_{k-1}, z_k)$ the  $(m+k)$-simplex they belong to. From
Proposition \ref{prop1} we have, since $\Delta(z_{k+1},
z_k)\subset\Delta,$
\begin{eqnarray*}
|f(z_{k-1})-f(z_k)|^2&\leq & C
\left(\mbox{diam}\,\Delta(z_{k-1},z_k)\right)^{d_w-d}\int_{\Delta(z_{k-1},z_k)}
\langle \nabla f, Z\, \nabla f\rangle d\nu\\
&\leq &
 C \left(\mbox{diam}\,\Delta(z_{k-1},z_k)\right)^{d_w-d} \int_{\Delta} \langle
\nabla f, Z\, \nabla f\rangle d\nu.
\end{eqnarray*}

Since $f$ is continuous, summing over $k$ we obtain
\begin{eqnarray*}
|f(z)-f(y)| & \leq & \sum_{k=1}^\infty |f(z_{k-1})-f(z_k)|\\
&\leq & \sum_{k=1}^\infty
\left(\mbox{diam}\,\Delta(z_{k-1},z_k)\right)^{\frac{d_w-d}{2}}\left(\int_\Delta
\langle \nabla f, Z
\nabla f\rangle d\nu\right)^{1/2}\\
&\leq &\sum_{k=1}^\infty L^{-\frac{m+k}{2}(d_w-d)}
\left(\int_\Delta \langle \nabla
f, Z\, \nabla f\rangle d\nu\right)^{1/2}\\
&=&
CL^{-\frac{m(d_w-d)}{2}}\left(\int_\Delta\grad\,d\nu\right)^{1/2}
\end{eqnarray*}
and consequently
\begin{equation}\label{hihi}
|f(z)-f(y)|^2\leq C L^{m(d-d_w)}\int_\Delta \langle\nabla f,
Z\,\nabla f\rangle d\nu.
\end{equation}
\medskip

\noindent\textsc{Step 2.} Suppose $x,y$ belong to a common
$m$-simplex $\Delta$. Then choose a vertex $v\in V(\Delta),$ write
$|f(x)-f(y)|^2\leq 2(|f(x)-f(v)|^2+|f(v)-f(y)|^2),$ and apply Step
1 in order to get \pref{hihi} for $x$ and $y.$

\medskip

\noindent \textsc{Step 3.} The result of Step 2 extends
immediately to the case when $x,y$ belong to two adjacent
$m$-simplices: when $x\in \Delta_1\in{\mathcal T}_m,$
$y\in\Delta_2\in{\mathcal T}_m$ and $\Delta_1, \Delta_2$ are
adjacent, then $\Delta_1$ and $\Delta_2$ share a vertex $z\in
V^{(m)}.$ One applies Step 1 to the pair $(x,z)$ and then to
$(y,z),$ getting
\begin{equation}\label{step3}
|f(x)-f(y)|^2\leq C L^{m(d-d_w)}\int_{\Delta_1\cup
\Delta_2}\langle \nabla f, Z\, \nabla f\rangle d\nu.
\end{equation}

\medskip

\noindent\textsc{Step 4.} Now take any $x,y\in \K\setminus
V^{(\infty)}$. Let $\mbox{ind}\,(x,y)=m.$ Then $S(x,y) =
\Delta_{m-1}(x)\cup\Delta_{m-1}(y)$ is composed either of a common
$(m-1)$-simplex or two adjacent $(m-1)$-simplices. In the first
case, apply   Step 2, in the latter case -- Step 3. In either
case, $\mu(S(x,y)) \asymp L^{-(m-1)d}$ and $\rho(x,y)\asymp
L^{-m},$ so the theorem is proven.
\end{proof}

\bigskip

{\noindent \em Proof of Theorem \ref{najwazniejsze}.} Choose
$\Delta\in {\mathcal T}_m.$  By Jensen's inequality we have
\[\barint_\Delta |f(x)-f_\Delta|d\mu(x)\leq \left(\barint_\Delta
|f(x)-f_\Delta|^2d\mu(x)\right)^{1/2},
\]
and further:
\begin{eqnarray*}
\barint_\Delta |f(x) - f_\Delta|^2 d\mu(x) & = &
\barint_\Delta |f(x) -\barint_\Delta f(y) d\mu(y)|^2 d\mu(x) \\
&=&
\barint_\Delta | \barint_\Delta (f(x) - f(y)) d\mu(y) |^2 d\mu(x) \\
&\le &
\barint_\Delta\barint_\Delta |f(x) - f(y)|^2 d\mu(y) d\mu(x) \\
&= & \frac1{\mu(\Delta)^2} \int_\Delta\int_\Delta |f(x) - f(y)|^2
d\mu(y) d\mu(x).
\end{eqnarray*}
Points $x$ and $y$ under the integral belong to a common
$m$-simplex $\Delta$, and so $\ind(x,y)> m$ (without loss of
generality we can and do assume that $x,y$ are nonvertex points).
Using Lemma \ref{ind-n}, we split the inner integral as follows.
\begin{eqnarray}\label{pp}
& &
\int_\Delta |f(x) - f(y)|^2 d\mu(y) \nonumber\\
& = & \sum_{n=m+1}^\infty \int_{\{y\in\Delta:\, \ind\,(x,y)=n \}}
|f(x) - f(y)|^2 d\mu(y)
\nonumber\\
& = & \sum_{n=m+1}^\infty \int_{(\Delta^*_{n-1}(x) \setminus
\Delta^*_n(x) )\cap \Delta} |f(x) - f(y)|^2 d\mu(y).
\end{eqnarray}
When $\ind(x,y)=n,$ then $\rho(x,y)\asymp L^{-n}$ and moreover
there exist two adjacent \mbox{$(n-1)$-simplices}, say $S$ and
$T,$ such that
 $x\in S,$ $y\in T$ $(S=T$ is permitted).

Let $v\in V^{(n-1)}$ be a common vertex of $S$ and $T.$ Then,
according to \pref{step3} (which is true without property {\bf
(P)} as well)
\begin{eqnarray*}
|f(x)-f(y)|^2 & \leq & CL^{-n(d_w-d)}\int_{S\cup T}\grad\,d\nu\\
&\leq & CL^{-n(d_w-d)}\int_{\Delta_{n-1}^*(x)}\grad\,d\nu.
\end{eqnarray*}

As $\mu(\Delta\cap(\Delta_{n-1}^*(x)\setminus \Delta^*_n(x)))\leq
\mu(\Delta_{n-1}^*(x)) \asymp L^{-nd}$, each of the integrals in
\pref{pp} is bounded by
\[CL^{-nd_w}\int_{\Delta_{n-1}^*(x)}\grad\,d\nu.\]

Consequently,
\begin{eqnarray}\label{here}
\int_\Delta \int_\Delta |f(x) - f(y)|^2 d\mu(y)d\mu(x) &\le& C
\sum_{n=m+1}^\infty L^{-nd_w} \int_\Delta \int_{\Delta^*_{n-1}(x)}
\grad d\nu d\mu(x).\nonumber\\
\end{eqnarray}

Let $w\in\W_m$ be such that $\Delta=\phi_w(\K)$ and for
$\underline{i}\in W_{n-1-m}$ set $\Delta_{\underline{i}} =
\phi_{w{\underline{i}}}(\K)\subset \Delta$. Observe that on each
$\Delta_{\underline{i}}$ the mapping $x\mapsto \Delta^*_{n-1}(x)$
is constant and equal to $\Delta^*_{\underline{i}}$. It follows
\begin{eqnarray}\label{ppp}
& & \int_\Delta \int_{\Delta^*_{n-1}(x)} \grad d\nu d\mu(x) \nonumber \\
&=&
\sum_{{\underline{i}}\in\W_{n-1-m}} \int_{\Delta_{\underline{i}}}\int_{\Delta^*_{\underline{i}}}\grad d\nu d\mu(x)\nonumber \\
&=& \sum_{{\underline{i}}\in\W_{n-1-m}}
\int_{\Delta^*_{\underline{i}}}\grad d\nu\mu(\Delta_{\underline{i}})\nonumber\\
&\leq & C \sum_{{\underline{i}}\in\W_{n-1-m}}
L^{-nd}\,\int_{\Delta^*_{\underline{i}}}\grad d\nu .
\end{eqnarray}

Sets $\Delta_{\underline{i}}^*$ are not pairwise disjoint, but
each of them is consists of at most $M+1$ simplices from
$\mathcal{T}_{n-1}$. Therefore, if in \eqref{ppp} we decompose
each of the integrals over $\Delta_{\underline{i}}^*$ into a
number of integrals over corresponding $(n-1)$-simplices, then
each of these $(n-1)$-simplices will appear at most $M+1$ times in
the sum. Furthermore, since for any ${\underline{i}}\in\W_{n-1-m}
$  one has $\Delta_{\underline{i}}^*\subset \Delta^*$, and
$\bigcup_{{\underline{i}}\in\W_{n-1-m}}\Delta_{\underline{i}}=\Delta,$
it follows that
\begin{equation}\label{pppi}
\sum_{i\in\W_{n-1-m}} \int_{\Delta^*_i}\grad d\nu \leq
C\int_{\Delta^*} \grad d\nu.
\end{equation}
Collecting \pref{here}, \pref{ppp}, \pref{pppi} we obtain
\begin{eqnarray*}
\int_\Delta\int_\Delta|f(x)-f(y)|^2d\mu(y)d\mu(x)&\leq& C
\sum_{n=m+1}^{\infty} L^{-nd_w}L^{-nd}\int_{\Delta^*}\grad d\nu\\
&= & C L^{-md_w}L^{-md}\int_{\Delta^*}\grad d\nu.
\end{eqnarray*}
To complete the proof, observe again that $L^{-md}=c \mu(\Delta).$
\phantom{} \hfill $\Box$

\medskip

Below we derive a Poincar\'e inequality that uses balls instead of
simplices. This statement requires property {\bf (P)} and will be
used throughout  for the results of next section.

\begin{theo}\label{poinc-balls} Suppose that $\K$ satisfies {\bf
(P)}. Suppose $f\in \mathcal{D}(\E)$. Let $x_0\in\K\setminus
\vinf$ be a nonvertex point and let $r>0$ be given.  Denote
$B=B(x_0,r) = \{ y\in\K: \rho(x_0,y)\le r\}$. Then there exist
$C>0$ and $A\geq 1$ (independent of $x_0$ and
 $r$) such that
\begin{equation}\label{eq-poinc-balls}
\barint_B |f-f_B|d\mu \leq C
r^{\frac{d_w}{2}}\left(\frac{1}{r^{d}}\int_{B(x_0, Ar)}\grad
d\nu\right)^{1/2}.
\end{equation}
\end{theo}

\begin{proof} Only minor changes need to be introduced  in the proof
of Theorem \ref{najwazniejsze}. From property {\bf (P)}  there
exists $\alpha\in(0,1)$ such that for every nonlattice $x\in\K,$
and any $m\geq 1$
\begin{equation}\label{popopo}
B(x,\frac{\alpha}{L^m})\subseteq \Delta_m^*(x)\subseteq B(x,
\frac{2}{L^m}).
\end{equation}
Let $n_0$ be the unique integer such that
$L^{-(n_0+1)}<\frac{r}{\alpha}\leq L^{-n_0},$ so that
\[B(x,r)\subseteq B(x,\alpha L^{-n_0})\subseteq \Delta_{n_0}^*(x).\]
As before, we get
\[
\barint_B |f-f_B|^2 d\mu \leq \frac{1}{\mu(B)^2} \int_B\int_B
|f(x)-f(y)|^2d\mu(x)d\mu(y).
\]
Since $B\subset \Delta^*_{n_0}(x_0)$ and $\Delta^*_{n_0}(x_0)=
S_1\cup \ldots \cup S_K$ is the sum of a finite number of
neighboring $n_0$-simplices, we estimate the inner integral as
\begin{eqnarray}\label{iii}
\int_{\Delta_0^*}|f(y)-f(x)|^2 d\mu(x) d\mu(y) & = & \sum_i
\int_{S_i}|f(x)-f(y)|^2d\mu(y).
\end{eqnarray}
Now we work with the integral over each $S_i$ separately. Observe
that when $x,y$ are as in the integral in \pref{iii}, then
$\Delta_{n_0}(x)\cap \Delta_{n_0}(y)\neq\emptyset,$ so that
$\mbox{{\rm ind}}\,(x,y) \geq n_0+1.$ Therefore, for any
$i=1\ldots K$, we have
\begin{eqnarray*}
\int_{S_i}|f(x)-f(y)|^2 d\mu(y) &= & \sum_{n=n_0+1}^\infty
\int_{S_i\cap \{y : \mathrm{ind}(x,y)=n \}}
|f(x)-f(y)|^2d\mu(y)\\
&=& \sum_{n=n_0}^\infty \int_{S_i\cap(\Delta^*_{n}(x)\setminus
\Delta_{n+1}^*(x))}
|f(x)-f(y)|^2 d\mu(y)\\
&\leq& c \sum_{n=n_0}^\infty L^{-nd_w}\int_{S_i\cap
\Delta^*_n(x)}\grad d\nu.
\end{eqnarray*}

From now on we proceed identically as in the proof of
\pref{poincfrac}, ending up with
\begin{eqnarray*}
\int_B\int_B |f-f_B|^2 d\mu d\mu
 & \leq & \int_B \sum_i\int_{S_i}
|f(x)-f(y)|^2d\mu(y)d\mu(x)\\
&\leq &  L^{-n_0d_w}L^{-n_0 d+f} \sum_i \int_{S_i^*} \grad d\nu\\
& \leq &
c L^{-n_0d}L^{-n_0 d}\int_{B(x_0,\frac{2L}{\alpha}r)} \grad d\nu \\
&\leq & c  r^{d_w} \frac{1}{\mu(B(x_0,
\frac{2L}{\alpha}r)}\int_{B(x_0, \frac{2L}{\alpha}r)}\grad d\nu,
\end{eqnarray*}
where we have used the inclusions
$S_i^*\subseteq\Delta_{n_0}^*\subseteq B(x_0,2L^{-n_0})\subseteq
B(x_0,\frac{2L}{\alpha}r)$. Set $A=\frac{2L}{\alpha}$. The proof
is complete.
\end{proof}

\section{Sobolev spaces on fractals}
On metric spaces, several definitions of Sobolev-type spaces are
possible (see e.g. \cite{FHK}, \cite{Hajl-met}, \cite{KS}). We
recall some of them below. Their mutual relations and connections
with the Poincar\'e inequality form now a well established theory
(\cite{Haj-Kos}, \cite{Hajl-met}). Below, we briefly recall the
relevant definitions.

Suppose $(X,\rho,\mu)$ is a  metric measure space, where $\mu$ is
a doubling Radon measure on a metric space $(X,\rho).$ Any nested
fractal $\mathcal K$ fits into this definition, with $\mu$ not
only doubling but even Ahlfors regular. In the following
definitions of Sobolev-type spaces we suppose $p\geq 1$.
\begin{enumerate}
\item The Hajlasz-Sobolev spaces $M^{1,p}(X)$ consists of those
functions $f\in L^p(X),$ for which there exists a function $g\in
L^p(X),$ $g\geq 0$ such that
\begin{equation}\label{hajl-sob}
|f(x)-f(y)|\leq C \rho(x,y) (g(x)+g(y))
\end{equation}
for $\mu$-almost all $x,y\in X.$

\item The space ${\mathcal P}^{1,p}(X)$ consists of those functions $f\in
L^1_{loc}(X)$, for which there exist $\sigma\ge 1$ and $g\in
L^p(X)$ such that for every ball $B=B(x,r)$
\begin{equation}\label{poinc-sob}
\barint_B|f-f_B|d\mu \leq r\left(\barint_{B(x,\sigma r)}
g^pd\mu\right)^{1/p}.
\end{equation}

\item The Korevaar-Schoen Sobolev space, $KS^{1,p}(X)$ consists of
those functions \linebreak $f\in L^p(X)$ for which
\begin{equation*}
\limsup_{\epsilon\to 0} \int_X \barint_{B(x,\epsilon)}
\frac{|f(x)-f(y)|^p}{\epsilon^p}\,d\mu(x)d\mu(y)<\infty.
\end{equation*}
\end{enumerate}
One considers also the Newtonian spaces $N^{1,p}(X)$.  The {\em
upper gradient} those spaces are based on   involves  integrals
over rectifiable curves. On nested fractals,  the family of
rectifiable curves might be empty or not rich enough to yield a
non-degenerate object.

In general, the inclusions $M^{1,p}(X)\subset {\mathcal
P}^{1,p}(X)\subset KS^{1,p}(X)$ hold true, but not always they can
be reversed. In some cases however  -- for example in
$\mathbb{R}^d$ -- all  three definitions yield the same function
spaces. We refer to \cite{Kos-MMan} and \cite{Hajl-met} for more
details.

We are now going to adapt definitions of the spaces $M^{1,p}$,
${\mathcal P}^{1,p}$ and $KS^{1,p}$ to the fractal setting. As we
have already mentioned in the Introduction, the scale $r$ is not a
natural scale here, and it will be replaced by
$r^{\frac{d_w}{2}}.$ Let us mention that in many cases (the
Euclidean spaces, some manifolds) the walk dimension $d_w$, read
off from the heat kernel estimates on the underlying space, is
equal to 2, so that the scale $r^{\frac{d_w}{2}}$ is just $r.$

\begin{defi}\label{sob-fractals} Let $\mathcal K$ be the nested fractal
defined in Section \ref{nested fractals}; let $p\geq 1$ and
$\sigma>0$ be given. Recall that $\mu$ denotes the normalized
$d$-dimensional Hausdorff measure on $\K$ and $\nu$ -- the Kusuoka
measure. We say that a function $f\in L^p(\K,\mu)$ belongs to:
\begin{itemize}
\item[--] the space $M^{1,p}_\sigma (\K,\mu),$ when there exists a
nonnegative function $g\in L^p(\K,\mu)$ such that for $\mu$-a.e.
$x,y\in \K$,
\begin{equation}\label{hsob-frac}
|f(x)-f(y)|\leq \rho(x,y)^\sigma (g(x)+g(y));
\end{equation}

\item[--] the space ${\mathcal P}^{1,p}_\sigma(\K)$, when there
exists a nonnegative function $g\in L^p(\K,\nu)$ such that for any
$x\in \K$ and $0<r<\diam \K$,
\begin{equation}\label{poinc-sob-frac} \barint_{B(x,r)} |f-f_{B(x,r)}|
d\mu\leq r^\sigma \left(\frac{1}{\mu(B(x,Ar)}\int_{B(x, Ar)
}g^pd\nu\right)^{1/p},
\end{equation}
with some $A\geq 1$; the inequality \eqref{poinc-sob-frac} will be
called the $(1,p,\sigma)-$Poincar\'e inequality;

\item[--] the space $KS^{1,p}_\sigma(\K),$ when
\begin{equation*}
\limsup_{\epsilon\to 0} \int_\K \barint_{B(x,\epsilon)}
\frac{|f(x)-f(y)|^p}{\epsilon^{p\sigma}}\,d\mu(x)d\mu(y)<\infty,
\end{equation*}
\item[--] the Besov-Lipschitz space $Lip(\sigma,p,\infty)$, $\sigma>0$
(see \cite{Jon1}), if
\[\|f\|_{Lip}=\sup_{m\geq 0} a_m^{(p)}(f)<\infty,\]
where
\[
a_m^{(p)}(f)=L^{m\sigma}\left( L^{md} \int\int_{\rho(x,y)\leq
\frac{c_0}{L^m}}|f(x)-f(y)|^pd\mu(x)d\mu(y)\right)^{1/p},
\]
with some $c_0>0$. Note that different values of this constant
yield the same function space with equivalent norms.
\end{itemize}
\end{defi}
It is immediate to see that the spaces $Lip(\sigma,p,\infty)(\K)$
and $KS_\sigma^{1,p}(\K)$ coincide and that their norms are
equivalent.

We now turn to relations between the Poincar\'{e}-Sobolev and
Korevaar-Schoen Sobolev spaces on fractals. The inclusion
 ${\mathcal
P}^{1,p}_\sigma(\K)\subset KS^{1,p}_\sigma(\K)$ is true under
usual constraints on parameters ($p\geq 1,$ $\sigma>d/p$), and it
can be reversed for $p=2,\sigma=\frac{d_w}{2}.$

\begin{prop}\label{poinc-to-ks}  Suppose that the fractal $\K$ satisfies property {\bf (P)}.
Let $p\geq 1$ and  $\sigma>0$ be given.
\begin{enumerate}
\item[(1)]
If $\sigma> d/p$, then ${\mathcal P}^{1,p}_\sigma(\K)\subset
KS^{1,p}_\sigma(\K).$
\item[(2)] When $\sigma=\frac{d_w}{2},$ then ${\cal P}^{1,2}_{\sigma}(\K)=KS^{1,2}_\sigma(\K).$
\end{enumerate}
\end{prop}

\begin{proof} Once (1) is proven, then the inclusion `$\subset$' in (2) would follow from the
relation $d_w>d$ (true for any nested fractal). As to the opposite
inclusion,   Theorem \ref{poinc-balls}  gives that the
$(1,2,\frac{d_w}{2})$--Poincar\'{e} inequality holds true for any
$f\in {\cal D}({\cal E})$. As ${\cal D} ({\cal E})=
Lip(\frac{d_w}{2}, 2,\infty)=KS^{1,2}_{d_w/2}(\K)$ (Theorem 5  of
\cite{KPP}), the inclusion `$\supset$' in (2) follows.

Therefore we need to prove (1).  Our proof is a modification of
the proof of Theorem 4.1 of \cite{Kos-MMan}. See also
\cite{Haj-Kos}, Theorem 5.3 and its proof.

Assume that $f\in {\mathcal P}^{1,p}_\sigma(\K)$ and that the pair
$(f,g)$ satisfies the $(1,p,\sigma)$-Poincar\'e inequality.
Introduce a fractal version of Riesz potentials:
\[
J_p(g,n,x)=\sum_{m=0}^\infty
L^{-(m+n)\sigma}\left(\frac{1}{\mu(\Delta^*_{n+m}(x))}
\int_{\Delta^*_{n+m}(x)}g^p(z)d\nu(z)\right)^{1/p}.
\]

The potentials $J_p(g,n,x),$ are well-defined for all nonlattice
points of $\K$ (this is a set of full measure $\mu$).

We will show that there exists a constant $k_0\geq 0$ such that
for $\mu$-a.a.  $x,y\in \K$ with $\ind(x,y)\geq k_0$ one has:
\begin{eqnarray}\label{riesz-maxi}
|f(x)-f(y)|&\leq& C\left( J_p(g,
\mbox{ind}\,(x,y)-k_0,x)+J_p(g,\mbox{ind}\,(x,y)-k_0,y)\right)
\end{eqnarray}
Since by assumption $f\in L^p(\K,\mu)\subset L^1(\K,\mu),$
$\mu$-almost every point of $\mathcal K$ is a $\mu$-Lebesgue point
for $f$ (cf. \cite{Tol}):
\[
f(x)=\lim_{r\to 0}\barint_{B(x,r)} f(y)d\mu(y)= \lim_{r\to
0}f_{B(x,r)}.
\]

Let $x,y$ be two nonlattice Lebesgue points for $f$ and let
$n_0=\ind(x,y)$. We use a classical chaining argument. Denote
$r_m=\frac{\alpha}{AL^m}$, where $A\geq 1$ is the constant from
the Poincar\'e inequality \pref{poinc-sob-frac}, and
$\alpha\in(0,1)$  comes  from \pref{popopo}. Using the Jensen's
inequality, the doubling property for $\mu$, the Poincar\'e
inequality \eqref{poinc-sob-frac} and  \pref{popopo}, we obtain
the following chain of inequalities:

\begin{eqnarray}\label{potrzebnedalej1}
|f(x)-f_{B(x,r_{n_0})}| &\leq & \sum_{m=0}^\infty
|f_{B(x_0,r_{n_0+m})}-f_{B(x,r_{n_0+m+1})}|\nonumber \\
&\leq& \sum_{m=0}^\infty \barint_{B(x,r_{n_0+m+1})}
|f(z)-f_{B(x,r_{n_0+m})}|d\mu(z)\nonumber \\
&\leq & \sum_{m=0}^\infty
\barint_{B(x,r_{n_0+m})}|f(z)-f_{B(x,r_{n_0+m})}|d\mu(z)\nonumber \\
&\leq & C\sum_{m=0}^\infty
r_{n_0+m}^\sigma\left(\frac{1}{\mu(B(x,Ar_{n_0+m}))}
\int_{B(x,Ar_{n_0+m})}g(z)^pd\nu(z)\right)^{1/p}\nonumber\\
&\leq & C\,\sum_{m=0}^\infty L^{-(m+n_0)\sigma}
\left(\frac{1}{\mu(\Delta_{n_0+m}^*(x))}
\int_{\Delta_{n_0+m}^*(x)} g(z)^pd\nu(z)\right)^{1/p}\nonumber \\
&=& C J_p(g,\mbox{{\rm ind}}\,(x,y),x).
\end{eqnarray}
Similar estimate holds for $y:$
\begin{eqnarray}\label{potrzebnedalej2}
|f(y)-f_{B(y,r_{n_0})}|&\leq & CJ_p(g,\ind(x,y),y).
\end{eqnarray}
From Lemma \ref{ind-n}, there exists a universal constant $C_1>0$
such that when $\ind(x,y)=n_0,$ then $\rho(x,y)\leq C_1 L^{-n_0}=
\frac{C_1 A}{\alpha}r_{n_0}.$ For short, denote
$R=(1+\frac{C_1A}{\alpha})r_{n_0}$. Let $k_0$ be the smallest
number such that for any $z\in\K$, $B(z,AR)\subset
\Delta^*_{n_0-k_0}(z)$, cf. \eqref{popopo}. Using the Poincar\'e
inequality \eqref{poinc-sob-frac} and the Ahlfors-regularity of
$\mu$ we get:
\begin{eqnarray}\label{potrzebnedalej3}
& &
|f_{B(x,r_{n_0})}-f_{B(y,r_{n_0})}| \\
& \leq &
|f_{B(x,r_{n_0})}-f_{B(x,R)}|+|f_{B(y,r_{n_0})}-f_{B(x,R)}|\nonumber \\
&\leq & \barint_{B(x,r_{n_0})} |f(z)-f_{B(x,R)}|\,d\mu(z)+
\barint_{B(y,r_{n_0})} |f(z)-f_{B(x,R)}|\,d\mu(z)\nonumber\\
&\leq & \left(\frac{\mu(B(x,R))}{\mu(B(x,r_{n_0}))}
+\frac{\mu(B(x,R))}{\mu(B(y,r_{n_0}))}\right)\barint_{B(x,R)}|f(z)-f_{B(x,R)}|
\,d\mu(z)\nonumber\\
&\leq &
CR^\sigma\left(\frac{1}{\mu(B(x,AR)}\int_{B(x,AR)}g(z)^pd\nu(z)\right)^{1/p}\nonumber\\
&\leq& CJ_p(g,\ind(x,y)-k_0,x).
\end{eqnarray}
The estimate \pref{riesz-maxi} follows when we sum up
\pref{potrzebnedalej1}, \pref{potrzebnedalej2},
\pref{potrzebnedalej3}.

The proposition  will be proven once we show that
\[\sup_{m\geq k_0}\left(a_m^{(p)}(f)\right)^p<\infty,\]
where
\[\left(a_m^{(p)}(f)\right)^p=L^{m(\sigma p+d)}\int\int_{\rho(x,y)\leq\frac{\alpha}{L^m}}
|f(x)-f(y)|^pd\mu(x)d\mu(y).\]
We have:
\begin{eqnarray}\label{dzisiaj1}
\nonumber & &
\left(a_m^{(p)}(f)\right)^p \\
& \le & \nonumber
\int_\K\int_{\Delta_m^*(x)} |f(x)-f(y)|^p d\mu(y)d\mu(x)\\
&\leq& \int_\K \left(\sum_{k=m+1}^\infty
\int_{\Delta_{k-1}^*(x)\setminus\Delta^*_k(x)} |f(x)-f(y)|^p
d\mu(y)\right)d\mu(x)
\end{eqnarray}
Since $y\in \Delta_{k-1}^*(x)\setminus\Delta^*_k(x)$ is tantamount
to $\ind(x,y)=k+1$, we can use the previously obtained estimate
\pref{riesz-maxi} and get
\begin{eqnarray}\label{dzisiaj2}
& & \int_{\Delta_{k-1}^*(x)\setminus\Delta^*_k(x)}|f(x)-f(y)|^p d\mu(y) \nonumber\\
&\leq & C\left(\int_{\Delta_{k-1}^*(x)\setminus\Delta^*_k(x)(x)}
  J_p^p(g,k-k_0,x)d\mu(y)+
\int_{\Delta_k^*(x)\setminus\Delta^*_{k+1}(x)}
  J_p^p(g,k-k_0,y)d\mu(y)\right)\nonumber\\
&=& C (I_k(x)+II_k(x)).
\end{eqnarray}
To estimate these two parts we need a lemma, which is similar to
Lemma 4.3 (ii), (iii) of \cite{Kos-MMan}:
\begin{lem}\label{kmm}
Let $N\geq 1,$ $p\geq 1, $ $\sigma>0$ be given and let the
functions $f\in L^p(\K,\mu),$  $g\in L^p(\K,\nu)$ satisfy the
$(1,p,\sigma)$-Poincar\'e inequality. Then for $\mu$-almost all
$x\in \K$
\begin{equation}\label{kmm1}
\int_{\Delta_N^*(x)} J_p^p(g,N,y) d\mu(y) \leq CL^{-N\sigma
p}\int_{\Delta_N^{**}(x)}g^p d\nu
\end{equation}
and
\begin{equation}\label{kmm2}
  \int_\K J_p^p(g,N,y)d\mu(y)\leq C\ L^{-N\sigma p}\int_\K g^p d\nu.
\end{equation}
\end{lem}

\begin{proof}
For $y\in\Delta_N^*(x)$ and $k\geq N$ one has
$\Delta_k^*(y)\subset \Delta_N^*(y)\subset\Delta_N^{**}(x)$ and
therefore
\begin{eqnarray*}
J_p(g,N,y)&=&\sum_{m=0}^\infty
L^{-(m+N)\sigma}\left(\frac{1}{\mu(\Delta_{N+m}^{*}(y))}\int_{\Delta_{N+m}^{*}(y)}
g^p(z)d\nu(z)\right)^{1/p}
\nonumber\\
&\leq & C\sum_{m=0}^\infty L^{-(\sigma-\frac{d}{p})(N+m)}
\left(\int_{\Delta_{N}^{**}(x)}g^p(z)d\nu(z)\right)^{1/p}\\
&=& CL^{-(\sigma-\frac{d}{p})N}\left(\int_{\Delta_N^{**}(x)}
g^p(z) d\nu(z)\right)^{1/p}.
\end{eqnarray*}
Since $\mu(\Delta^*_N(x)) \leq C L^{-Nd},$  \pref{kmm1} follows.

To see \pref{kmm2}, observe that, using \pref{kmm1}:
\begin{eqnarray*}
\int_\K J_p^p(g,N,y) d\mu(y)&=&\sum _{\Delta\in {\mathcal
T}_N}\int_\Delta J_p^p(g,N,y)d\mu(y) \\
&\leq & C \sum_{\Delta\in {\mathcal T}_N}
L^{-Np\sigma}\int_{\Delta^{**}} g^pd\nu.
\end{eqnarray*}
A covering argument as the one used to conclude the proof of
Theorem \ref{najwazniejsze} gives \pref{kmm2}.
\end{proof}

\noindent {\em Conclusion of the proof of Proposition
\ref{poinc-to-ks}.} Since
\[
I_k(x)\leq \mu(\Delta_{k-1}^*(x))J_p^p(g,k-k_0,x)\leq
CL^{-kd}J_p^p(g,k-k_0,x),
\]
one has, using \pref{kmm2}
\begin{eqnarray}\label{pierwsze}
\int_\K I_k(x)d\mu(x) &\leq &
\int_\K\int_{\Delta^*_k(x)}J_p^p(g,k-k_0,x)d\mu(y)d\mu(x)
\nonumber\\ & \leq &
C\int_\K J_p^p(g,k-k_0,x)\mu(\Delta_{k}^*(x))d\mu(x)\nonumber\\
&\leq & CL^{-k(d+\sigma p)}\int_\K g^pd\nu.
\end{eqnarray}
To estimate the other part, we use \pref{kmm1}:
\begin{eqnarray}\label{drugie}
    \int_\K II_k(x)d\mu(x)&\leq &
    \int_\K\int_{\Delta^*_k(x)}J_p^p(g,k-k_0,y)d\mu(y)d\mu(x)\nonumber\\
    &\leq & C\int_\K\int_ {\Delta^*_{k-k_0}(x)} J_p^p(g,k-k_0,y)d\mu(y)d\mu(x)\nonumber\\
    &\leq & C \int_\K L^{-k\sigma p}\int_{\Delta_{k-k_0}^{**}(x)}g^pd\nu\,d\mu(x)\nonumber\\
    &\leq & CL^{-k(d+\sigma p)}\int_\K g^pd\nu.
\end{eqnarray}
Summing up \pref{pierwsze} and \pref{drugie} over $k\geq m$ we get
that the right-hand side of \pref{dzisiaj1} is not bigger than
\begin{eqnarray*}
    C\sum_{k=m}^\infty L^{-k(d+\sigma p)}\int_\K g^pd\nu
    &=& CL^{-m(d+\sigma p)}\int_\K g^pd\nu,
\end{eqnarray*}
so that $$\left(a^{(p)}_m(f)\right)^p\leq \int_\K g^pd\nu,$$ once
$m\geq k_0.$ The proposition follows.
\end{proof}

We now turn our attention to the relation of Poincar\'{e}-Sobolev
spaces ${\cal P}^{1,p}_\sigma({\K})$ to  Haj{\l}asz-Sobolev spaces
$M^{1,p}_\sigma (\K).$ It has been proven  by Hu that
 $M^{1,p}_\sigma(\K)\subset
KS^{1,p}_\sigma({\mathcal K}),$ for all $p\geq 1$ and $\sigma>0 $
(Theorem 1.1. of \cite{Hu}).  Moreover, this theorem asserts that
one has the inclusion $KS^{1,p}_\sigma({\mathcal K})\subset
M^{1,p}_{\sigma'}(\K),$ for all $0<\sigma'<\sigma.$ It is not
known whether  the inclusion  $KS^{1,p}_\sigma({\mathcal
K})\subset M^{1,p}_{\sigma}(\K)$ holds true on nested fractals,
even if we assume that Property {\bf (P)} holds.

Recall that  for $p\geq 1,$ the 'weak' $L^p,$ or the {\em
Marcinkiewicz space} $L^p_w(\K,\mu)$ consists of those measurable
functions $f$ for which
\[
\sup_{t>0}\{t^p \,\mu \{x: |f(x)|>t \} \}< +\infty.
\]
We can consider the `weak' Hajlasz-Sobolev spaces.
\begin{defi} Let $p\geq 1 $ and $\sigma >0.$
One says that $f\in L^p(\K,\mu)$ belongs to the weak
Hajlasz-Sobolev space $(M^{1,p}_\sigma)_w(\K)$,  if there exists
$g\in L^p_w(\K,\mu)$ such that \pref{hajl-sob} holds true.
\end{defi}

\noindent
 We have the following.
\begin{prop}\label{weak-hajl} Suppose that the nested fractal $\K$ satisfies
Property {\bf (P)}. Assume $p\geq 1, $ $\sigma>0$. Then one has:
\begin{enumerate}
\item[(1)]
$ {\mathcal P}^{1,p}_{\sigma}(\K) \subset (M^{1,p}_\sigma)_w(\K)
\subset M^{1,p'}_\sigma(\K),$ with any $1\leq p'<p $ (the last
inclusion requires $p>1$).
\item[(2)] When $p=2,$ $\sigma=d_w/2,$ then $M^{1,2}_\sigma(\K)\subset {P}^{1,2}_\sigma({\K}). $
\end{enumerate}
\end{prop}
\begin{proof} (1) Once we have proven estimates for fractal Riesz-potentials,
this result is immediate. Let $f\in {\mathcal
P}^{1,p}_\sigma(\K),$ and let $(f,\tilde f)$ satisfy the
$(1,p,\sigma)$ Poincar\'{e} inequality.

The function $g$ (corresponding to the upper gradient), needed in
the definition of Haljasz-Sobolev spaces,  will be a fractal
variant of the Hardy-Littlewood maximal function: for $x\in
\K\setminus V^{(\infty)}$ we set
\begin{eqnarray*}
g(x)=(M\tilde f)(x)&\stackrel{def}{=}& \sup_{m\geq 1}
\left(\frac{1}{\mu(\Delta_m^*(x))} \int_{\Delta_m^*(x)}  \tilde
f^p d\nu \right)^{1/p}.
\end{eqnarray*}
It is obvious that for any $n\geq 1$
\begin{equation}\label{october1}
J_p(\tilde f,n,x)\leq C L^{-n\sigma}
 g(x),
\end{equation}
with some universal constant $C>0.$ Recall the estimate
\pref{riesz-maxi}:
\[|f(x)-f(y)|\leq  C\left( J_p(\tilde f,
\mbox{ind}\,(x,y)-k_0,x)+J_p(\tilde f
,\mbox{ind}\,(x,y)-k_0,y)\right)
\]
($k_0$ was a universal index depending only on the geometry of the
fractal), so that further, taking into account the relation
\pref{ind-dist}
\[|f(x)-f(y)|\leq C L^{-\sigma \ind(x,y)} (g(x)+g(y))\leq C \rho(x,y)^\sigma (g(x)+g(y)).\]

 The argument that proves  $g\in
L^p_w(\K,\nu)$ is also classical. Fix $t>0$ and suppose that
$g(x)>t$ for some $x\in\K\setminus\vinf$. By the definition of
$g$, there exists $m=m(x)$ such that
\begin{equation}\label{jutro1}
  \mu(\Delta_m^*(x))\leq \frac{1}{t^p}\int_{\Delta_m^*(x)}
  \tilde f^p  d\nu.
\end{equation}
Consider the covering of the set $A(t)=\{x\in \K: g(x)>t\}$ by
balls $B(x,2 L^{-m(x)})$, $x\in A(t)\setminus\vinf$. By the
$5r$-covering lemma there is a countable subcollection of these
balls, $B_i=B(x_i, \rho_i)$, with $\rho_i=2L^{-m(x_i)}$, such that
the $B_i$'s are pairwise disjoint, yet  $A(t) \subset \bigcup_i
B(x_i, 5\rho_i)$. Due to \pref{popopo}, the sets
$\Delta^*_{m(x_i)}(x_i)$ are disjoint. Then, by the doubling
property of $\mu,$
\begin{eqnarray*}
\mu(\{x: g(x)>t\} &\leq& \mu\left(\bigcup_i B(x_i,5\rho_i)\right)
\leq
C\sum_i\mu(B(x_i,\rho_i))\\
&\leq& C\sum_i \mu(B(x_i, \alpha L^{-m(x_i)})) \leq
C\sum_i\mu(\Delta^*_{m(x_i)})\\
&\stackrel{\pref{jutro1}}{\leq}& \frac{C}{t^p}\sum_i
\int_{\Delta^*_{m(x_i)}} \tilde f^p d\nu \leq \frac{C}{t^p}\int_\K
\tilde f^p d\nu.
\end{eqnarray*}  Since $\mu(\K)<\infty, $ we have $L^p_w(\K,\mu)\subset L^{p'}(\K,\mu)$ for
$p'<p.$ This way (1)
is proven. Assertion (2) follows from Hu's inclusion
$M^{1,2}_\sigma(\K)\subset KS^{1,p}_\sigma(\K)$ and Proposition
\ref{poinc-to-ks} (2) above.

\end{proof}

}
\section{Appendix}

We will now prove the statement from Remark \ref{num}.  Set
\begin{eqnarray*}\alpha_0 =\inf\{\mbox{dist}\,(A,B): A,B\in{\cal T}_2,
A\cap B=\emptyset\} &\mbox{  and } &\alpha = L\alpha_0.
\end{eqnarray*}

{More precisely, we will be proving the following.}

\begin{prop}\label{ap1} Let $\K$ be the nested fractal
associated with similitudes $\{\phi_i\}_{i=1}^r$ with contraction
factor $L.$ Suppose that the $\phi_i$'s share their unitary parts,
i.e. there is an isometry $U:\rn\to\rn$ such that
$\phi_i(x)=\frac{1}{L} \,U(x)+t_i,$ $t_i\in\rn,$ $i=1,2,...,r.$
Then {\bf (P)} holds.
\end{prop}

 The key argument in the proof is provided by the following lemma.
\begin{lem}\label{adjacent}  Let $n\geq 1.$
Suppose $A,B$ are two neighbouring $n-$simplices, and let $A_1\subset A, $  $B_1\subset B $
be two $(n+1)-$simplices that are disjoint.
Then $\mbox{dist}\,(A_1,B_1)\geq \alpha L^{-n}$.
\end{lem}

\begin{proof}
We proceed by induction on $n.$ Clearly, the statement is true for $n=1$.

Suppose that the statement is true for $1,...,n-1.$ Let
$A,B\in{\cal T}_n,$ $ A_1,B_1\in {\cal T}_{n+1}$ be as in the
statement; let $(i_1,...,i_n)$ be  the address of $A$ and
$(w_1,.,,,w_n)$ -- the address of $B.$ Define $k_0=\min\{l:
i_l\neq w_l\}.$ One has $1\leq k_0\leq n.$

If  $k_0>1$ then $A,B\subset \K_{i_1...i_{k_0-1}}\in {\cal T}_{k_0-1}.$
Set
\begin{eqnarray*}
A'=\phi_{i_1...i_{k_0-1}}^{-1}(A), &&         B'=\phi_{i_1...i_{k_0-1}}^{-1}(B) \;\;\;\;
(\mbox{we have }\; A',B'\in {\cal T}_{n-k_0+1}),   \\
A'_1=\phi_{i_1...i_{k_0-1}}^{-1}(A_1), && B_1'=\phi_{i_1...i_{k_0-1}}^{-1}(B_1)\;\;
(\mbox{we have }\; A_1',B_1'\in {\cal T}_{n-k_0+2}).
\end{eqnarray*}
Those simplices satisfy the assumptions for $n-k_0+1\leq (n-1))$
and the statement follows.

Now, suppose that $k_0=1$. We have
\begin{eqnarray*}
  \K_{i_1}\supset \K_{i_1i_2} \supset...\supset \K_{i_1...i_n}
    &=& A\supset A_1= \K_{i_1...i_ni_{n+1}}, \\
  \K_{w_1}\supset \K_{w_1w_2} \supset...\supset \K_{w_1...w_n}
    &=& B\supset B_1= \K_{w_1...w_nw_{n+1}}.
\end{eqnarray*}
Let $v$ be a junction point of $A$ and $B$. Because of the
inclusions above, \linebreak $v\in  \K_{i_1i_2}\cap \K_{w_1w_2}
\subset \K_{i_1}\cap \K_{w_1}$ as well.

We will now show that $\K_{i_1i_2}\cup\K_{w_1w_2}$ is similar to $\K_{i_1}\cup\K_{w_1}.$
More precisely, we will see that
\begin{equation}\label{star}
\K_{i_1i_2} - v = S(\K_{i_1}-v) \qquad \text{and} \qquad
\K_{w_1w_2} - v = S(\K_{w_1}-v),
\end{equation}
where $S=\frac{1}{L} U$ is the similitude such that $\phi_i=S+t_i.$

Since $v\in \K_{i_1}\cap\K_{w_1}\subset V^{(1)}$, there
exist $z_1, z_2\in V^{(0)}$ and mappings $\phi_{j_1}$,
$\phi_{j_2}$ such that $z_l$ is the fixed point of $\phi_{j_l}$,
$l=1,2$, and $v=\phi_{i_1}(z_1)=\phi_{w_1}(z_2)$. Further, since
$v\in\K_{i_1i_2}$, there exist another essential fixed point $u$,
 such that $v=\phi_{i_1i_2}(u)$. Then we have $\phi_{i_1}(z_1) =
\phi_{i_1i_2}(u)$ and so $z_1 = \phi_{i_2}(u)$. In particular,
$z_1\in\K_{i_2}\cap\K_{j_1}$. By Proposition IV.13 of \cite{Lin},
any element in $\vo$ belongs to exactly one $n$-cell for each $n$.
It follows that $j_1=i_2$. The same argument for $\K_{w_1w_2}$
gives $j_2=w_2$.

Since $S$ is linear, we have
\[ S(K_{i_1}-v)=S(\phi_{i_1}(K)-\phi_{i_1}(z_1))= S(SK+t_{i_1}-Sz_1-t_{i_1}) = S^2(K-z_1).\]
On the other hand, since $z_1=\phi_{i_2}(z_1)$, we get
\begin{align*}
  K_{i_1i_2}-v &= \phi_{i_1}(\phi_{i_2}(K))-\phi_{i_1}(z_1) \\
  &= S(\phi_{i_2}(K))-S(z_1) \\
  &= S(\phi_{i_2}(K)-\phi_{i_2}(z_1)) \\
  &= S(SK-Sz_1)\\
  &= S^2(K-z_1).
\end{align*}
Identical arguments hold for the pair $\K_{w_1}$ and $\K_{w_1w_2}$
and the  proof of \pref{star} is complete.

Now,
\begin{eqnarray*}
A'=S^{-1}(A-v)+v,\;\;&&\;\; B'=S^{-1}(B-v)+v,\\
A_1'=S^{-1}(A_1-v)+v,\;\;&&\;\; B_1'=S^{-1}(B_1-v)+v,
\end{eqnarray*}
are two pairs of $(n-1)-$ and $(n-2)-$simplices satisfying the
assumptions, hence $\mbox{dist}\, (A_1', B_1')\geq
\frac{\alpha}{L^{n-1}},$ and thus
$\mbox{dist}\,(A_1,B_1)\geq\frac{\alpha}{L^{n}}.$ This completes
the proof.
\end{proof}

\noindent {\em Proof of Proposition \ref{ap1}.} We proceed by
induction on $n.$

If $n=1$ and $y\notin \Delta_1^*(x)\setminus \Delta_2^*(x)$,
then the $2-$simplices $\Delta_2(x)$ and $\Delta_2(y)$ are disjoint. Thus,
$\rho(x,y)\geq\mbox{dist}\,(\Delta_2(x),\Delta_2(y)) \geq\alpha_0= \alpha/L.$

Suppose now that the statement is true for $1,2,...,n-1,$ and
take\linebreak $y\in\Delta_n^*(x)\setminus \Delta_{n+1}^*(x).$
Then the sets $\Delta_{n+1}(x)$ and $\Delta_{n+1}(y)$ are
disjoint, whereas $\Delta_n(x)$ and $\Delta_n(y)$ are not. There
are two possibilities: either $\Delta_n(x)=\Delta_n(y),$ or they
are adjacent $n-$simplices. Let $(i_1,...,i_n)$ be the address of
$\Delta_n(x)$ and $(w_1,...,w_n)$ be the address of $\Delta_n(y).$

If $\Delta_n(x)=\Delta_n(y)$, we consider points
$x'=\phi_{i_1...i_{n-1}}^{-1}(x)$, $y'=\phi_{i_1...i_{n-1}}^{-1}(y).$
Then $\Delta_2(x')$ and $\Delta_2(y')$ are disjoint
$2-$simplices, so from the assumption we get $\rho(x',y')\geq \frac{\alpha}{L},$
thus $\rho(x,y)\geq \frac{\alpha}{L^n}.$

If $\Delta_n(x)$ and $\Delta_n(y)$ are adjacent $n-$simplices, we
apply lemma \ref{adjacent} to $A=\Delta_n(x),$ \linebreak
$B=\Delta_n(y),$ $A_1=\Delta_{n+1}(x),$ $B_1=\Delta_{n+1}(y).$
\hfill$\Box$\\

{\bf Acknowledgements.} The work of K.P.P. was supported  by
the Polish Ministry of Science grant no. N N201 397837 (years
2009-2012). Part of this research was conducted while K.P.P. was
visiting Univerit\'{e} Blaise Pascal, Clermont-Ferrand. The author
wants to thank UBP for its hospitality. A.S. was partially
supported by the MNiSW grant N N201 373136.

\end{document}